\documentclass[12pt,reqno]{amsart}
\usepackage[cp1251]{inputenc}
\usepackage[T2A]{fontenc}
\usepackage[english]{babel}
\usepackage{geometry}
\usepackage{amsmath, amsthm, amscd, amsfonts, amssymb, graphicx, color}
\usepackage[bookmarksnumbered, colorlinks, plainpages]{hyperref}

\numberwithin{equation}{section}

\newtheorem{theorem}{Theorem}[section]

\newtheorem{proposition}[theorem]{Proposition}

\theoremstyle{remark}
\newtheorem{remark}[theorem]{Remark}

\theoremstyle{definition}
\newtheorem{definition}[theorem]{Definition}
\newtheorem*{main-definition}{Main Definition}

\begin{document}

\title[Unconditional convergence of eigenfunction expansions]{Unconditional convergence of eigenfunction expansions for abstract and elliptic operators}


\author[V. Mikhailets]{Vladimir Mikhailets}

\address{Institute of Mathematics of the Czech Academy of Sciences, \v{Z}itn\'{a} 25, Praha, 11567, Czech Republic; Institute of Mathematics of the National Academy of Sciences of Ukraine, Tereshchen\-kivs'ka 3, Kyiv, 01024, Ukraine}

\email{vladimir.mikhailets@gmail.com; mikhailets@imath.kiev.ua}


\author[A. Murach]{Aleksandr Murach}

\address{Institute of Mathematics of the National Academy of Sciences of Ukraine, Tereshchen\-kivs'ka 3, Kyiv, 01024, Ukraine}

\email{murach@imath.kiev.ua}






\begin{abstract}
We study the most general class of eigenfunction expansions for abstract normal operators with pure point spectrum in a complex Hilbert space. We find sufficient conditions for such expansions to be unconditionally convergent in spaces with two norms and also estimate the degree of this convergence. Our result essentially generalizes and complements the known theorems of M.~Krein and of Krasnosel'ski\u{\i} and Pustyl'nik. We apply it to normal elliptic pseudodifferential operators on compact boundaryless $C^{\infty}$-manifolds. We find generic conditions for eigenfunction expansions induced by such operators to converge unconditionally in the Sobolev spaces $W^{\ell}_{p}$ with $p>2$ or in the spaces $C^{\ell}$ (specifically, for the $p$-th mean or uniform convergence on the manifold). These conditions are sufficient and necessary for the indicated convergence on Sobolev or H\"ormander function classes and are given in terms of parameters characterizing these classes. We also find estimates for the degree of the convergence on such function classes. These results are new even for differential operators on the circle and for multiple Fourier series.
\end{abstract}

\maketitle

\section{Introduction}\label{sec1}

Let $(X,\|\cdot\|)$ be a complex or real Banach space. A series $\sum_{n=1}^{\infty}x_{n}$ formed by elements $x_{n}\in X$ is called unconditionally convergent in $X$ if the rearranged series $\sum_{n=1}^{\infty}x_{\sigma(n)}$ converges in $X$ for an arbitrary permutation $\sigma:\mathbb{N}\leftrightarrow\mathbb{N}$. In this case there exists a unique element $x\in X$ such that $\sum_{n=1}^{\infty}x_{\sigma(n)}=x$ for every permutation $\sigma$. If
$\sum_{n=1}^{\infty}\|x_{n}\|<\infty$, then the original series is called absolutely convergent in $X$. Every absolutely convergent series in $X$ converges unconditionally therein. If the space is finite-dimensional, then the converse is true according to Riemann's theorem. However, if $X$ is infinite-dimensional, then there exist unconditionally convergent series in $X$ that do not converge absolutely (as the Dvoretzky--Rogers theorem states \cite{DvoretzkyRogers50}).

In this paper we study series in abstract and function spaces with two norms. The weaker norm among the two is induced by an inner product. Terms of series are pairwise orthogonal and are eigenfunctions of a certain (generally, unbounded) normal operator. Specifically, it may be a self-adjoint operator with discrete spectrum.

The simplest example of such type is given by the trigonometric Fourier series. From the operator-theoretic viewpoint, this is the spectral expansion of a function corresponding to the Laplace operator on the unit circle $\mathbb{T}$. The study of the convergence of such series has a very long history and huge literature (see, e.g., monographs \cite{Bary64, Edwards79-82, Zygmund02}, surveys \cite{AlimovIlinNikishin76, AlimovIlinNikishin77, Golubov84}, and references therein).

It is known that the trigonometric Fourier series of every function $f(\tau)$ from the Lebesgue space $L_{p}(\mathbb{T})$, with $1<p<\infty$, converges to $f$ in this space. This is true for the Fourier series with respect to any rearranged trigonometric system
\begin{equation}\label{rearranged-system}
f(\tau)\sim\frac{a_{0}}{2}+
\sum_{n=1}^{\infty}\bigl[a_{\sigma(n)}\cos(\sigma(n)\tau)+
b_{\sigma(n)}\sin(\sigma(n)\tau)\bigr]
\end{equation}
if and only if $p=2$. Moreover, there exists a certain permutation $\sigma$ and a continuous function $f$ on $\mathbb{T}$ such that series \eqref{rearranged-system} does not converges to $f$ in the space   $L_{p}(\mathbb{T})$ whatever $p>2$ \cite[Theorem~8]{Uljanov58}. Let us indicate another example. The Carleson theorem \cite{Carleson66} states that the Fourier trigonometric series of any function $f\in L_{2}(\mathbb{T})$ converges to $f$ almost everywhere (a.e.) on $\mathbb{T}$. However, as was shown by Kolmogorov \cite{KolmogoroffMenchoff27}, there exists a function $f\in L_{2}(\mathbb{T})$ such that its certain rearranged Fourier series of the form \eqref{rearranged-system} diverges a.e. on $\mathbb{T}$. Thus, conditions for function series to be convergent or unconditionally convergent can be essentially different.

The study of solutions to initial-boundary-value problems for parabolic or hyperbolic differential equations in cylindrical domains by the Fourier method leads to the analysis of the convergence of expansions in eigenfunctions of self-adjoint elliptic differential operators in bounded Euclidean domains (see, e.g., \cite{Ilin60, Ladyzhenskaya53, Maurin72}). Let $G$ be a bounded domain in $\mathbb{R}^{n}$ with smooth boundary, and suppose that an elliptic differential expression with smooth coefficients on $\overline{G}$ induces a self-adjoint operator $L$ with discrete spectrum in the Hilbert space $L_{2}(G)$. Then this space possesses an orthonormal basis $\{e_{j}(\tau):j\in\mathbb{N}\}$ formed by eigenfunctions of $L$. Choosing a function $f\in L_{2}(G)$ arbitrarily, we consider its spectral expansion
\begin{equation}\label{intr-spectral-expansion}
f(\tau)\sim\lim_{\lambda\to\infty}
\sum_{j:|\lambda_{j}|\leq\lambda}c_{j}(f)\,e_{j}(\tau).
\end{equation}
Here, $\lambda_{j}$ is the eigenvalue of $L$ associated with the eigenfunction $e_{j}$, and $c_{j}(f)$ is the Fourier coefficient of $f$ with respect to $e_{j}$. This expansion converges in the $L_{2}(G)$ norm. Therefore, it is  reasonable to ask under which supplementary conditions for $f$ the expansion \eqref{intr-spectral-expansion} converges to $f$ in the stronger $L_{p}(G)$ norm with $p>2$ or uniformly on $G$. There are no comprehensive answers to this question yet even in the case of ordinary differential equations. Results obtained in the many-dimensional case show that it is necessary to narrow down the problem statement essentially and to consider the convergence of the Fourier series \eqref{intr-spectral-expansion} not for single functions but on the classes of differential functions with compact supports in the domain $G$. Let us formulate one complete result of this type, which became a starting point for us (see, e.g., \cite{AlimovIlinNikishin76, AlimovIlinNikishin77}).

Let $W^{s}_{p,0}(G)$ denote the class of all functions that belong to the  $L_{p}$-Sobolev space of real order $s>0$ over $G$ and that their supports are compact subsets of the open set $G$. The spectral expansion \eqref{intr-spectral-expansion} of an arbitrary function $f\in W^{s}_{p,0}(G)$ converges uniformly on every compact set $K\subset G$ if and only if the following three conditions are satisfied:
\begin{equation}\label{intr-3-conditions}
ps>n,\quad s\geq\frac{n-1}{2},\quad\mbox{and}\quad p\geq1.
\end{equation}
Note that the first inequality guarantees the continuity functions of class $W^{s}_{p,0}(G)$, whereas the second inequality provides the validity of the localization principle for spectral expansions on the class of elliptic operators under consideration. If $p\leq2n/(n-1)$, then the second inequality in \eqref{intr-3-conditions} follows from the first one. The claim for the compactness of the support of $f$ is stipulated by the fact that then $f$ satisfies certain homogeneous boundary conditions depending on the elliptic operator $L$. The uniform convergence of the spectral expansions only on compact subsets of $G$ is connected with the presence of the boundary $\partial G$.

The main results of this paper are formulated and proved in Sections \ref{sec4} and~\ref{sec5}. There we study the unconditional convergence of expansions in eigenfunctions of an arbitrary normal elliptic pseudodifferential operator of positive order on a closed smooth manifold~$\Gamma$. The convergence is considered with respect to the norm in the Sobolev space $W^{\ell}_{p}(\Gamma)$ or in the space $C^{\ell}(\Gamma)$, with $p>2$ and $0\leq\ell\in\mathbb{Z}$. We find necessary and sufficient conditions for this unconditional convergence to hold true on various classes of differentiable functions and estimate its degree. To this end, we use the H\"ormander function classes $H^{\alpha}(\Gamma)$ whose differentiation order is a function $\alpha:[1,\infty)\to(0,\infty)$ that O-regularly varies at infinity in the sense of Avakumovi\'c. Such classes and their connection with elliptic pseudodifferential operators are discussed in Section~\ref{sec3}.

Sections \ref{sec6} and~\ref{sec7} specify these results for differential operators on the circle and for multiple Fourier series, which allows us to obtain a number of new valuable results for these objects.

Our proofs of the aforementioned results are based on an abstract theorem about the unconditional convergence (see Definition~\ref{def-convergence}) of expansions in eigenvectors of a normal operator in a space with two norms. This theorem is stated and proved in Section~\ref{sec2}. It develops and refines known theorems of M.~Krein \cite[Theorem~4]{Krein47} and of Krasnosel'ski\u{\i} and Pustyl'nik \cite[Theorem~22.1]{KrasnoselskiiZabreikoPustylnikSobolevskii76} with respect to normal operators with pure point spectrum in a Hilbert space.

Final remarks and comments to the obtained results are given in Section~\ref{sec8}.

\section{The basic abstract result}\label{sec2}

Throughout the paper, we consider complex linear spaces. Let $H$ be an infinite-dimensional (not necessarily separable) Hilbert space, and let $N$ be a normed space. Suppose that $H$ and $N$ are algebraically embedded in a certain linear space. As usual, $\|\cdot\|_{H}$ and $(\cdot,\cdot)_{H}$ respectively denotes the norm and inner product in~$H$, and $\|\cdot\|_{N}$ stands for the norm in~$N$. Thus, the linear space $H\cap N$ is endowed with two norms.

Let $L$ be a normal (specifically, self-adjoint) linear operator in $H$. We allow the case where $L$ is unbounded in $H$ as well as the case where  $L$ is bounded on $H$. We suppose that $L$ has pure point spectrum, i.e. the Hilbert space $H$ has an orthonormal basis $\{e_{j}:j\in\Theta\}$ formed by certain eigenvectors $e_{j}$ of $L$. Here, $\Theta$ is an index set whose cardinality coincides with the Hilbert dimension of $H$. Thus, every vector $f\in H$ expands in the series
\begin{equation}\label{f-series}
f=\sum_{j\in\Theta}(f,e_j)_{H}\,e_j
\end{equation}
in the topology of $H$. (Of course, this series contains only a countable number of nonzero terms, and its sum does not depend on their order.) Let $\lambda_{j}$ denote the eigenvalue of $L$ such that $Le_j=\lambda_{j}e_j$. As usual, $\sigma(L)$ stands for the spectrum of $L$, and $\sigma_{p}(L)$ denotes the point spectrum of $L$, i.e. $\sigma_{p}(L)$ is the set of all eigenvalues of $L$. Note that  eigenvalues of $L$ may be of infinite geometric multiplicity and that the set of all eigenvalues of $L$ may have accumulation points.

Let $\mathcal{Q}$ denote the system of all finite subsets of $\Theta$. We interpret $\mathcal{Q}$ as an upper directed set with respect to the relation $\subseteq$. The equality \eqref{f-series} means for $f\in H$ that the net
\begin{equation}\label{net}
\biggl\{\sum_{j\in\Upsilon}(f,e_j)_{H}\,e_j:\Upsilon\in\mathcal{Q}\biggr\}
\end{equation}
converges to $f$ in $H$.

\begin{definition}\label{def-convergence}
Let $f\in H\cap N$. We say that the expansions of $f$ in eigenvectors of $L$ \emph{converge unconditionally} in the normed space $N$ if the net \eqref{net} lies in $N$ and converges to $f$ in $N$ for every orthonormal basis $\{e_{j}:j\in\Theta\}$ in $H$ formed by eigenvectors of $L$. If this condition is satisfied for every vector $f$ from some class $M\subseteq H\cap N$, we also say that the expansions in eigenvectors of $L$ converge unconditionally in $N$ \emph{on the class} $M$.
\end{definition}

\begin{remark}
Providing $H$ is separable, the eigenvector expansion \eqref{f-series} becomes
\begin{equation}\label{f-series-separable}
f=\sum_{j=1}^{\infty}(f,e_j)_{H}\,e_j.
\end{equation}
In this case, the expansions of $f$ in eigenvectors of $L$ converge unconditionally in $N$ if and only if the series \eqref{f-series-separable} converges to $f$ in $N$ for every orthonormal basis $\{e_{j}:j\in\mathbb{N}\}$ formed by eigenvectors of $L$ (i.e. for any choices of orthonormal bases of eigenspaces of $L$ associated with non-degenerate eigenvalues). It follows from this condition that the series \eqref{f-series-separable} converges under an arbitrary permutation of its terms.
\end{remark}

\begin{theorem}\label{th-basic-series}
Let $\omega,\eta:\sigma(L)\to\mathbb{C}\setminus\{0\}$ be Borel measurable bounded functions, and define the bounded operators $\omega(L)$ and $\eta(L)$ on $H$ as functions of the normal operator~$L$. Assume that the mapping $u\mapsto\omega(L)u$ sets a bounded operator from $H$ to $N$, and denote this operator by~$R$. Let $f$ be an arbitrary vector from the image of the operator $(\omega\eta)(L)$. Then the expansions of $f$ in eigenvectors of $L$ converge unconditionally in the space~$N$. Moreover, the degree of this convergence admits the estimate
\begin{equation}\label{th-estimate-series}
\biggl\|f-\sum_{j\in\Upsilon}(f,e_j)_{H}\,e_j\biggr\|_{N}\leq
\|R\|_{H\to N}\cdot\|g\|_{H}\cdot
\sup_{j\in\Theta\setminus\Upsilon}|\eta(\lambda_j)|
\cdot r_{g,e}(\Upsilon)
\end{equation}
for each set $\Upsilon\in\mathcal{Q}$ and every orthonormal basis $e:=\{e_{j}:j\in\Theta\}$ indicated in Definition~$\ref{def-convergence}$  and with some net $\{r_{g,e}(\Upsilon):\Upsilon\in\mathcal{Q}\}\subseteq[0,1]$ that tends monotonically to zero and depends neither on $\omega$ nor on~$\eta$. Here, $g\in H$ is the unique vector such that $(\omega\eta)(L)g=f$.
\end{theorem}

\begin{proof}
Since $\omega(L)\eta(L)e_j=(\omega\eta)(\lambda_j)e_j$ whenever $j\in\Theta$ and since $(\omega\eta)(t)\neq0$ whenever $t\in\sigma(L)$, we conclude that each $e_j\in N$ because the image of $\omega(L)$ lies in $N$. Thus, the left-hand side of \eqref{th-estimate-series} makes sense. Since the operator $(\omega\eta)(L)$ is algebraically reversible, the vector $g$ is well defined for every above-mentioned $f$. Suppose that $f\neq0$, otherwise the conclusion of this theorem is trivial. Given $\Upsilon\in\mathcal{Q}$, we let $P_{\Upsilon}$ denote the orthoprojector of $H$ on the span of $\{e_j:j\in\Upsilon\}$ and obtain the following estimate:
\begin{equation}\label{proof-est1}
\begin{aligned}
\|f-P_{\Upsilon}f\|_{N}=&
\|(\omega\eta)(L)g-P_{\Upsilon}(\omega\eta)(L)g\|_{N}\\
=&\|(\omega\eta)(L)(I-P_{\Upsilon})g\|_{N}=
\|\omega(L)\eta(L)(I-P_{\Upsilon})^{2}g\|_{N}\\
\leq&\|R\|_{H\to N}\cdot\|\eta(L)(I-P_{\Upsilon})\|_{H\to H}\cdot \|(I-P_{\Upsilon})g\|_{H},
\end{aligned}
\end{equation}
with $I$ being the identity operator on $H$. We put
\begin{equation}\label{proof-def-r}
r_{g,e}(\Upsilon):=\|(I-P_{\Upsilon})g\|_{H}\cdot\|g\|_{H}^{-1}
\end{equation}
and see that the net $\{r_{g,e}(\Upsilon):\Upsilon\in\mathcal{Q}\}$ is required. Note also that
\begin{equation}\label{proof-est2}
\|\eta(L)(I-P_{\Upsilon})\|_{H\to H}\leq
\sup_{j\in\Theta\setminus\Upsilon}|\eta(\lambda_j)|.
\end{equation}
Indeed, since
\begin{equation*}
\eta(L)(I-P_{\Upsilon})h=
\eta(L)\sum_{j\in\Theta\setminus\Upsilon}(h,e_{j})_{H}\,e_{j}=
\sum_{j\in\Theta\setminus\Upsilon}(h,e_{j})_{H}\,\eta(\lambda_j)e_{j}
\end{equation*}
for every $h\in H$ (the convergence holds in $H$), we have the inequality
\begin{equation*}
\|\eta(L)(I-P_{\Upsilon})h\|_{H}^{2}=
\sum_{j\in\Theta\setminus\Upsilon}|(h,e_{j})_{H}\,\eta(\lambda_j)|^{2}\leq
\|h\|_{H}^{2}\cdot\sup_{j\in\Theta\setminus\Upsilon}|\eta(\lambda_j)|^{2},
\end{equation*}
which gives \eqref{proof-est2}. Now \eqref{proof-est1}--\eqref{proof-est2} yield the estimate \eqref{th-estimate-series}, which implies the required unconditional convergence.
\end{proof}

\begin{remark}\label{rem-comp-norms}
Assume that the norms in $H$ and $N$ are compatible on the linear space $H\cap N$. (By definition \cite[Chapter~I, Section~2.2]{GelfandShilov58}, this compatibility means the following: if a sequence lies in $H\cap N$ and converges to zero in one of these norms and is a Cauchy sequence with respect to the second norm, then it converges to zero in the second norm.) Let $R$ be a bounded linear operator on $H$ (specifically, $R:u\mapsto\omega(L)u$ where $u\in H$, as in Theorem~\ref{th-basic-series}). Then the inclusion $R(H)\subseteq N$ implies that $R$ acts continuously from $H$ to~$N$. Indeed, suppose that $R(H)\subseteq N$. Then, owing to the closed graph theorem, the operator $R:H\to\widetilde{N}$ is bounded if and only if it is closed; here, $\widetilde{N}$ is the completion of the normed space~$N$. We therefore have to prove that this operator is closable. Assume that a sequence $(u_{k})_{k=1}^{\infty}\subset H$ satisfies the following two conditions: $u_{k}\to0$ in $H$ and
\begin{equation}\label{rem6.3-convegence}
Ru_{k}\to v\;\;\mbox{in}\;\;\widetilde{N}\;\;\mbox{for certain}\;\;
v\in\widetilde{N},
\end{equation}
as $k\to\infty$. Then $Ru_{k}\to0$ in $H$ because $R$ is bounded on~$H$. Besides, $(Ru_{k})_{k=1}^{\infty}$ is a Cauchy sequence in $N$. Hence, $Ru_{k}\to0$ in $N$ because the norms in $N$ and $H$ are compatible on $H\cap N$. Thus, $v=0$ in view of \eqref{rem6.3-convegence}, which means that the operator $R:H\to\widetilde{N}$ is closable.
\end{remark}

Theorem~\ref{th-basic-series} contains M.~Krein's result \cite[Theorem~4]{Krein47} according to which the series \eqref{f-series-separable} converges in $N$ for every $f\in L(H)$ if $L$ is a compact self-adjoint operator in $H$ and if $L$ acts continuously from $H$ to a Banach space $N$. The latter theorem generalizes (to abstract operators) the Hilbert\,--\,Schmidt theorem about the uniform decomposability of sourcewise representable functions in eigenfunctions of a symmetric integral operator.

If $L$ is a positive definite self-adjoint operator with discrete spectrum and if $\omega(t)\equiv t^{-\delta}$ and $\eta(t)\equiv t^{-\tau}$ for certain numbers $\delta,\tau\geq0$ and if $R:=L^{-\delta}$ acts continuously from $H$ to a Banach space $N$, Krasnosel'ski\u{\i} and Pustyl'nik \cite[Theorem~22.1]{KrasnoselskiiZabreikoPustylnikSobolevskii76} proved that
\begin{equation*}
\biggl\|f-\sum_{j=1}^{k}(f,e_j)_{H}\,e_j\biggr\|_{N}=
o(\lambda_{k}^{-\tau})\;\;\,\mbox{as}\;\;k\to\infty
\end{equation*}
for any $f$ from the domain of $L^{\delta+\tau}$ (the Hilbert space $H$ is assumed to be separable). Here, the eigenvalues $\lambda_{k}>0$ of $L$ are numbered so that $\lambda_{k}\leq\lambda_{k+1}$. Formula \eqref{th-estimate-series} implies a more general result for normal operators with discrete spectrum.

\section{H\"ormander spaces and elliptic operators}\label{sec3}

We will use them to estimate the degree of the convergence of expansions in eigenfunctions of some elliptic operators given on a closed manifold. Such spaces were introduced over $\mathbb{R}^{n}$ and investigated by L.~H\"ormander \cite[Section 2.2]{Hermander63}. Their smoothness index (or differentiation order) is a general enough function of the frequency variables, which allows us to obtain finer estimates than those \cite[Subsection 6.1\,a]{Agranovich94} received  with the help of classical function spaces such as Sobolev spaces, whose smoothness index is a number. We use a broad class of isotropic inner product H\"{o}rmander spaces that admit a reasonable definition on smooth closed manifolds. The class was introduced and investigated in \cite{MikhailetsMurach13UMJ3}, \cite[Section 2.4.2]{MikhailetsMurach14}, and \cite{MikhailetsMurach15ResMath1}. It has various applications to elliptic operators \cite[Section 2.4.3]{MikhailetsMurach14} and elliptic boundary-value problems \cite{AnopDenkMurach21, AnopChepurukhinaMurach21Axioms}.

This class is formed by certain function spaces $H^{\alpha}$ whose smoothness index $\alpha$ ranges over a set denoted by $\mathrm{OR}$. By definition, $\mathrm{OR}$ consists of all Borel functions
$\alpha:\nobreak[1,\infty)\rightarrow(0,\infty)$ such that
\begin{equation*}
c^{-1}\leq\frac{\alpha(\lambda t)}{\alpha(t)}\leq c
\quad\mbox{whenever}\;\;t\geq1\;\;\mbox{and}\;\;1\leq\lambda\leq b,
\end{equation*}
with the numbers $b>1$ and $c\geq1$ being independent of $t$ and $\lambda$ (but may depend on $\alpha$). Such functions were introduced by V.~G.~Avakumovi\'c \cite{Avakumovic36}, are called OR-varying (or O-regularly varying) at infinity, and are well investigated \cite{BinghamGoldieTeugels89, BuldyginIndlekoferKlesovSteinebach18}.

The class $\mathrm{OR}$ admits the following description \cite[Theorem 2.2.7]{BinghamGoldieTeugels89}: $\alpha\in\mathrm{OR}$ if and only if
\begin{equation}\label{OR-description}
\alpha(t)=
\exp\Biggl(\beta(t)+
\int\limits_{1}^{t}\frac{\gamma(\tau)}{\tau}\;d\tau\Biggr)
\quad\mbox{whenever}\;\;t\geq1
\end{equation}
for some bounded Borel functions $\beta,\gamma:[1,\infty)\to\mathbb{R}$.

For every function $\alpha\in\mathrm{OR}$ there exist numbers $s_{0},s_{1}\in\mathbb{R}$, with $s_{0}\leq s_{1}$, and numbers $c_{0},c_{1}>0$ such that
\begin{equation}\label{RO-inequalities}
c_{0}\lambda^{s_{0}}\leq\frac{\alpha(\lambda t)}{\alpha(t)}\leq
c_{1}\lambda^{s_{1}}\quad\mbox{whenever}\quad t\geq1\;\;\mbox{and}\;\;\lambda\geq1.
\end{equation}
Given $\alpha\in\mathrm{OR}$, we let $s_{\ast}(\alpha)$ denote the supremum of the set of all numbers $s_{0}$ that the left-hand inequality in \eqref{RO-inequalities} holds true (for certain $c_{0}$ depending on $s_{0}$), and we let $s^{\ast}(\alpha)$ denote the infimum of the set of all numbers $s_{1}$ that the right-hand inequality in \eqref{RO-inequalities} holds true (for certain $c_{1}$ depending on $s_{1}$). Of course, $-\infty<s_{\ast}(\alpha)\leq s^{\ast}(\alpha)<\infty$. The numbers $s_{\ast}(\alpha)$ and $s^{\ast}(\alpha)$ are equal to the lower and upper Orlicz--Matuszewska indices of $\alpha$, resp. These facts follow from \cite[Theorem 2.1.7 and Proposition 2.2.1]{BinghamGoldieTeugels89}.

For instance, every continuous function $\alpha:[1,\infty)\rightarrow(0,\infty)$ such that
\begin{equation*}
\alpha(t)=t^{s}(\log t)^{s_{1}}(\log\log t)^{s_{2}}\ldots (\underbrace{\log\ldots\log}_{k\;\mbox{\tiny{times}}} t)^{s_{k}}\quad\mbox{whenever}\quad t\gg1
\end{equation*}
belongs to $\mathrm{OR}$; here, $1\leq k\in\mathbb{Z}$ and $s,s_{1},\ldots,s_{k}\in\mathbb{R}$. It is well known that  $s_{\ast}(\alpha)=s^{\ast}(\alpha)=s$ for this function.

An example of a function $\alpha\in\mathrm{OR}$ with the different Orlicz--Matuszewska indices is given by formula \eqref{OR-description} in which
\begin{equation}\label{example-gamma}
\gamma(\tau):=\left\{
\begin{array}{ll}
r&\hbox{if}\;\;\tau\in[\theta_{1}\cdots\theta_{2j-1},
\theta_{1}\cdots\theta_{2j-1}\theta_{2j}]\;\,\hbox{for certain integer}\;j\geq1,\\
s&\hbox{otherwise}
\end{array}\right.
\end{equation}
provided that $r<s$, $1=\theta_{1}<\theta_{2}<\theta_{3}<\dots$, and $\theta_{k}\to\infty$ as $k\to\infty$. Then $s_{\ast}(\alpha)=r$ and $s^{\ast}(\alpha)=s$, as is noticed, e.g., in \cite[Section~3]{AnopChepurukhinaMurach21Axioms}.

Let $\alpha\in\mathrm{OR}$ and $1\leq n\in\mathbb{Z}$. By definition, the linear space $H^{\alpha}(\mathbb{R}^{n})$ consists of all tempered distributions $w$ on $\mathbb{R}^{n}$ that their Fourier transform $\widehat{w}:=\mathcal{F}w$ is locally Lebesgue integrable over $\mathbb{R}^{n}$ and satisfies the condition
$$
\int\limits_{\mathbb{R}^{n}}\alpha^{2}(\langle\xi\rangle)\,
|\widehat{w}(\xi)|^{2}\,d\xi<\infty.
$$
As usual, $\langle\xi\rangle:=(1+|\xi|^{2})^{1/2}$ is the smoothed absolute value of $\xi\in\mathbb{R}^{n}$. The space $H^{\alpha}(\mathbb{R}^{n})$ is endowed with the inner product
$$
(w_{1},w_{2})_{H^{\alpha}(\mathbb{R}^{n})}:=
\int\limits_{\mathbb{R}^{n}}\alpha^{2}(\langle\xi\rangle)\,
\widehat{w_{1}}(\xi)\,\overline{\widehat{w_{2}}(\xi)}\,d\xi
$$
and the corresponding norm. This space is an isotropic Hilbert case of the H\"ormander spaces $\mathcal{B}_{p,k}$ considered in \cite[Section 2.2]{Hermander63} and \cite[Section 10.1]{Hermander83}. Namely, $H^{\alpha}(\mathbb{R}^{n})=\mathcal{B}_{2,k}$ provided that $k(\xi)=\alpha(\langle\xi\rangle)$ whenever $\xi\in\mathbb{R}^{n}$.

Consider a version of the space $H^{\alpha}(\mathbb{R}^{n})$ for smooth closed manifolds. In the sequel, $\Gamma$ is a closed (i.e. compact and boundaryless) manifold of dimension $n\geq1$ and class $C^{\infty}$. Suppose that a positive $C^{\infty}$-density $dx$ is given on $\Gamma$. We arbitrarily choose a finite atlas from the $C^{\infty}$-structure on $\Gamma$; let this atlas be formed by $\varkappa$ local charts $\pi_{j}: \mathbb{R}^{n}\leftrightarrow \Gamma_{j}$, with $j=1,\ldots,\varkappa$.
Here, the open sets $\Gamma_{1},\ldots,\Gamma_{\varkappa}$ form a covering of $\Gamma$. We also arbitrarily choose functions $\chi_j\in C^{\infty}(\Gamma)$, with $j=1,\ldots,\varkappa$, such that $\chi_{1}(x)+\cdots+\chi_{\varkappa}(x)=1$ whenever $x\in\Gamma$ and that $\mathrm{supp}\,\chi_j\subset \Gamma_j$ whenever $1\leq j\leq\varkappa$. Thus, these functions form a partition of unity on $\Gamma$ subordinate to the above covering.

By definition, the linear space $H^{\alpha}(\Gamma)$ consists of all distributions $f$ on $\Gamma$ such that $(\chi_{j}f)\circ\pi_{j}\in H^{\alpha}(\mathbb{R}^{n})$ for each $j\in\{1,\ldots,\varkappa\}$. Here, $(\chi_{j}f)\circ\pi_{j}$ is the representation of the distribution $\chi_{j}f$ in the local chart $\pi_{j}$. The space $H^{\alpha}(\Gamma)$ is endowed with the inner product
\begin{equation*}
(f_{1},f_{2})_{H,\alpha}:=
\sum_{j=1}^{\varkappa}\,((\chi_{j}f_{1})\circ\pi_{j},
(\chi_{j}f_{2})\circ\pi_{j})_{H^{\alpha}(\mathbb{R}^{n})}
\end{equation*}
and the corresponding norm $\|f\|_{H,\alpha}:=(f,f)_{H,\alpha}^{1/2}$. This space is complete (i.e. Hilbert) and separable and does not depend up to equivalence of norms on the choice of the atlas and partition of unity on $\Gamma$; the set $C^{\infty}(\Gamma)$ is dense in $H^{\alpha}(\Gamma)$. These properties are due to \cite[Theorem~2.21]{MikhailetsMurach14}.

If $\alpha(t)\equiv t^{s}$ for some $s\in\mathbb{R}$, then $H^{\alpha}(G)$, where $G=\mathbb{R}^{n}$ or $G=\Gamma$, becomes the inner product Sobolev space $W^{s}_{2}(G)$ of order $s$.

We also need an equivalent definition of the space $H^{\alpha}(\Gamma)$, with $\alpha\in\mathrm{OR}$, in terms of a function of some elliptic pseudodifferential operators (PsDOs). Following \cite[Section~2.1]{Agranovich94}, we let $\Psi_{\mathrm{ph}}^{m}(\Gamma)$ denote the class of all classical (i.e. polyhomogeneous) PsDOs on $\Gamma$ of order $m\in\mathbb{R}$. Suppose that $m>0$ and that a PsDO $A\in\Psi_{\mathrm{ph}}^{m}(\Gamma)$ is elliptic on $\Gamma$, i.e. $a_{0}(x,\xi)\neq0$ whenever $x\in\Gamma$ and $0\neq\xi\in T_{x}^{\ast}\Gamma$, with $a_{0}(x,\xi)$ denoting the principal symbol of $A$ and with $T_{x}^{\ast}\Gamma$ standing (as usual) for the cotangent space to $\Gamma$ at $x$. We may and will consider $A$ as a closed unbounded operator in the Hilbert Lebesgue space $L_{2}(\Gamma):=L_{2}(\Gamma,dx)$
and with the domain $\mathrm{Dom}\,A=W^{m}_{2}(\Gamma)$ (see \cite[Theorem~2.3.5]{Agranovich94}). Suppose that $A$ is a self-adjoint operator in $L_{2}(\Gamma)$ and that $\sigma(A)\subseteq[1,\infty)$. Then the operator $\alpha(A^{1/m})$ is well defined in the Hilbert space $L_{2}(\Gamma,dx)$ via Spectral Theorem as the Borel function $\alpha(t^{1/m})$ of $A$.

\begin{proposition}\label{prop-equiv-def}
Let $\alpha\in\mathrm{OR}$. Then the norm in the space $H^{\alpha}(\Gamma)$ is equivalent to the norm
\begin{equation}\label{equivalent-norm}
f\mapsto\|\alpha(A^{1/m})f\|_{L_{2}(\Gamma)}
\end{equation}
on the dense set $C^{\infty}(\Gamma)$. Thus, $H^{\alpha}(\Gamma)$ coincides with the completion of $C^{\infty}(\Gamma)$ with respect to the norm \eqref{equivalent-norm}. Hence, if the function $1/\alpha$ is bounded on $[1,\infty)$, then $H^{\alpha}(\Gamma)$ is the domain of the operator
$\alpha(A^{1/m})$ and this operator sets an isomorphism between $H^{\alpha}(\Gamma)$ and $L_{2}(\Gamma)$.
\end{proposition}

This proposition is due to \cite[Theorem~2.23]{MikhailetsMurach14}. We may put $A:=1-\Delta_{\Gamma}$ and $m=2$, where $\Delta_{\Gamma}$ is the Laplace--Beltrami operator on $\Gamma$, with $\Gamma$ being endowed with the Riemannian metric inducing $dx$.

\section{Mean convergence}\label{sec4}

Given $s\in\mathbb{R}$ and $p\in(1,\infty)$, we let $W^{s}_{p}(\Gamma)$ denote the Sobolev space over $\Gamma$ with the smoothness index $s$ and the integral-exponent $p$ and let $\|\cdot\|_{W,s,p}$ stand for the norm in this space. The space $W^{s}_{p}(\Gamma)$ is defined on the base of  $W^{s}_{p}(\mathbb{R}^{n})$ with the help of the above atlas and partition of unity like the definition of $H^{\alpha}(\Gamma)$ on the base of $H^{\alpha}(\mathbb{R}^{n})$. Recall that the Sobolev space $W^{s}_{p}(\mathbb{R}^{n})$ consists of all tempered distributions $w$ on $\mathbb{R}^{n}$ such that the distribution
$v:=\mathcal{F}^{-1}[\langle\xi\rangle^{s}\,\widehat{w}(\xi)]$ belongs to the Lebesgue space $L_{p}(\mathbb{R}^{n})$, with the norm of $w$ in $W^{s}_{p}(\mathbb{R}^{n})$ being equal to the norm of $u$ in $L_{p}(\mathbb{R}^{n})$. Of course, $\mathcal{F}^{-1}$ stands for the inverse Fourier transform. Note that $W^{0}_{p}(\Gamma)$ coincides with the Lebesgue space $L_{p}(\Gamma):=L_{p}(\Gamma,dx)$, and assume that this holds with equality of norms.

We use the term "mean convergence"\ to refer to the convergence in the normed Sobolev space $W^{\ell}_{p}(\Gamma)$ subject to $0\leq\ell\in\mathbb{Z}$ and $1<p<\infty$, specifically in the Lebesgue space $L_{p}(\Gamma)$.

Let real $m>0$, and let $A\in\Psi_{\mathrm{ph}}^{m}(\Gamma)$. Suppose that the PsDO $A$ is elliptic on $\Gamma$. If $A$ is a normal operator in $L_{2}(\Gamma)$, the separable Hilbert space $L_{2}(\Gamma)$ has an orthonormal basis formed by some eigenfunctions of $A$ (see, e.g., \cite[Section~15.3]{Shubin01}), every eigenfunction of $A$ pertaining to $C^{\infty}(\Gamma)$ and each eigenvalue of $A$ being of finite multiplicity. Then every vector $f\in L_{2}(\Gamma)$ admits the spectral decomposition
\begin{equation}\label{spectral-decomp}
f=\lim_{r\to\infty}
\sum_{\substack{\lambda\in\sigma_{p}(L)\\|\lambda|\leq r}}P(\lambda)f
\end{equation}
the series converging in $L_{2}(\Gamma)$. As usual, $P(\lambda)$ denotes the orthoprojector on the eigenspace of $A$ associated with the eigenvalue~$\lambda$. Note that the sum in \eqref{spectral-decomp} contains only a finite number of terms for every $r>0$. The convergence in $L_{2}(\Gamma)$ implies one in each normed space $L_{p}(\Gamma)$ with $1\leq p<2$; hence, we are interested in conditions for the convergence of \eqref{spectral-decomp} in the mean with index $p>2$.

Given real $m>0$, we let $\mathrm{NE}\Psi_{\mathrm{ph}}^{m}(\Gamma)$ denote the set of all PsDOs of class $\Psi_{\mathrm{ph}}^{m}(\Gamma)$ that are elliptic on $\Gamma$ and normal in $L_{2}(\Gamma)$. Specifically, every elliptic constant-coefficient partial differential operator on $\Gamma$ of order $m$ belongs to $\mathrm{NE}\Psi_{\mathrm{ph}}^{m}(\Gamma)$. Recall that $n=\dim\Gamma\geq1$.

\begin{theorem}\label{th-mean}
Suppose that $0<m\in\mathbb{R}$, $0\leq\ell\in\mathbb{Z}$, $s,p,q\in\mathbb{R}$, and $1<q\leq2<p$. Then the following three conditions are equivalent:
\begin{itemize}
  \item[(i)] $s\geq\ell+n/q-n/p$;
  \item[(ii)] there exists a PsDO $A\in\mathrm{NE}\Psi_{\mathrm{ph}}^{m}(\Gamma)$ that the spectral decomposition \eqref{spectral-decomp} induced by $A$ converges in $W^{\ell}_{p}(\Gamma)$ on the class $W^{s}_{q}(\Gamma)$;
  \item[(iii)] the expansions in eigenfunctions of an arbitrary PsDO from $\mathrm{NE}\Psi_{\mathrm{ph}}^{m}(\Gamma)$ converge unconditionally in $W^{\ell}_{p}(\Gamma)$ on the class $W^{s}_{q}(\Gamma)$.
\end{itemize}
\end{theorem}

\begin{proof}
We will first prove that $\mathrm{(i)}\Rightarrow\mathrm{(iii)}$. Assuming (i) to hold true, we arbitrarily choose a PsDO $A\in\mathrm{NE}\Psi_{\mathrm{ph}}^{m}(\Gamma)$ and put $L:=(I+A^{\ast}A)$, with $I$ denoting the identity operator on  $L_{2}(\Gamma)$. Let $\sigma:=\ell+n/2-n/p>0$. Since $L\in\mathrm{NE}\Psi_{\mathrm{ph}}^{2m}(\Gamma)$ and $\sigma(L)\subseteq[1,\infty)$, the operator $L^{\sigma/(2m)}$ is well defined in $L_{2}(\Gamma)$ as a power function of $L$ and sets an isomorphism between $W^{\sigma}_{2}(\Gamma)$ and $L_{2}(\Gamma)$; see \cite[Corollary~5.3.2]{Agranovich94}. The inverse of $L$, i.e. the operator $L^{-\sigma/(2m)}$, acts continuously from $H=L_{2}(\Gamma)$ to $N:=W^{\ell}_{p}(\Gamma)$ because of the continuous embedding $W^{\sigma}_{2}(\Gamma)\hookrightarrow W^{\ell}_{p}(\Gamma)$. This embedding follows directly from its well-known analog for Sobolev spaces over $\mathbb{R}^{n}$ (see, e.g., \cite[Theorem 2.8.1(b)]{Triebel95}). Using Theorem~\ref{th-basic-series}, we put $\omega(t):=t^{-\sigma/(2m)}$ and $\eta(t):=1$ whenever $t\geq1$ and remark that $W^{\sigma}_{2}(\Gamma)$ is the image of the operator $(\omega\eta)(L)$. We conclude by this theorem that the expansions in eigenfunctions of $L$ converge unconditionally in $W^{\ell}_{p}(\Gamma)$ on the class $W^{\sigma}_{2}(\Gamma)$. Observe that $W^{s}_{q}(\Gamma)\hookrightarrow W^{\sigma}_{2}(\Gamma)$ because $s\geq\sigma+n/q-n/2$ by (i). This yields (iii) since all eigenfunctions of $A$ are eigenfunctions of $L$ as well.
We have proved that $\mathrm{(i)}\Rightarrow\mathrm{(iii)}$.

The implication $\mathrm{(iii)}\Rightarrow\mathrm{(ii)}$ is obvious.

It remains to prove that $\mathrm{(ii)}\Rightarrow\mathrm{(i)}$. Assuming (ii) to hold true, we obtain the embedding of $W^{s}_{q}(\Gamma)$ in $W^{\ell}_{p}(\Gamma)$. The embedding implies (i), which is considered to be  a known fact. Specifically, this fact follows from \cite[p.~60, property (ii)]{Triebel06}. Namely, the above embedding implies that  $W^{s}_{q}(G)\hookrightarrow W^{\ell}_{p}(G)$ for any open ball $G$ in $\mathbb{R}^{n}$. (Recall that $W^{s}_{q}(G)$, e.g., consists of the restrictions of all distributions $w\in W^{s}_{q}(\mathbb{R}^{n})$ to $G$.) If (i) was not true, i.e. $s+\varepsilon=\ell+n/q-n/p$ for some $\varepsilon>0$, then the Besov space $B^{s+\varepsilon}_{q,\theta}(G)$ (as a part of $W^{s}_{q}(G)$) would be embedded in $W^{\ell}_{p}(G)$ for each $\theta\geq1$. However, by \cite[p.~60, property (ii)]{Triebel06}, the last embedding is equivalent to $\theta\leq p$. This contradiction shows that  $\mathrm{(ii)}\Rightarrow\mathrm{(i)}$.
\end{proof}

Let $m>0$, and let $A\in\mathrm{NE}\Psi_{\mathrm{ph}}^{m}(\Gamma)$. Consider an orthonormal basis $e:=(e_{j})_{j=1}^{\infty}$ in $L_{2}(\Gamma)$ formed by eigenfunctions $e_{j}\in C^{\infty}(\Gamma)$ of $A$. Let $\lambda_{j}$ denote the eigenvalue of $A$ such that $Ae_j=\lambda_{j}e_j$. We enumerate the eigenfunctions so that $|\lambda_{j}|\leq|\lambda_{j+1}|$ whenever $j\geq1$. Then
\begin{equation}\label{eigenvalue-asymptotic}
|\lambda_{j}|\sim\widetilde{c}\,j^{m/n}\quad\mbox{as}\quad j\to\infty,
\end{equation}
where $\widetilde{c}$ is a certain positive number that does not depend on~$j$ (see, e.g., \cite[Section~15.3]{Shubin01}). According to Theorem~\ref{th-mean}, the series \eqref{f-series-separable}, where $H=L_{2}(\Gamma)$, converges unconditionally in $W^{\ell}_{p}(\Gamma)$ on the class $W^{\ell+n/2-n/p}_{2}(\Gamma)$. This class is the broadest one among the spaces $W^{s}_{q}(\Gamma)$ indicated by Theorem~\ref{th-mean} as classes of convergence of the above series. Using the H\"ormander spaces $H^{\alpha}(\Gamma)\subset W^{\ell+n/2-n/p}_{2}(\Gamma)$ as classes of the convergence, we can estimate its degree.

\begin{theorem}\label{th-mean-degree}
Let $0\leq\ell\in\mathbb{Z}$, $2<p\in\mathbb{R}$, and $\alpha\in\mathrm{OR}$. Suppose that
\begin{equation*}
h(t):=t^{\ell+n/2-n/p}\,(\alpha(t))^{-1}\to0
\quad\mbox{as}\quad t\to\infty
\end{equation*}
and that the function $h$ decreases on $[1,\infty)$. Then
\begin{equation}\label{mean-degree-est-lambda}
\biggl\|f-\sum_{j:|\lambda_{j}|\leq\lambda}
(f,e_j)_{H}\,e_j\biggr\|_{W,\ell,p}
\leq c\cdot\|f\|_{H,\alpha}\cdot h(\lambda^{1/m})
\end{equation}
and
\begin{equation}\label{mean-degree-est}
\biggl\|f-\sum_{j=1}^{k}(f,e_j)_{H}\,e_j\biggr\|_{W,\ell,p}
\leq c\cdot\|f\|_{H,\alpha}\cdot h(k^{1/n})
\end{equation}
for all $f\in H^{\alpha}(\Gamma)$, $\lambda\geq1$, and integer-valued $k\geq1$. Here, $c$ is a certain positive number that does not depend on $f$, $\lambda$, and $k$.
\end{theorem}

\begin{proof}
Put $\sigma:=l+n/2-n/p$, $L:=(I+A^{\ast}A)$, and $\omega(t):=t^{-\sigma/(2m)}$ and $\eta(t):=h(t^{1/(2m)})$ whenever $t\geq1$. Since $(\omega\eta)(t)=(\alpha(t^{1/(2m)}))^{-1}$ whenever $t\geq1$, the operator $(\omega\eta)(L)$ sets an isomorphism between $L_{2}(\Gamma)$ and $H^{\alpha}(\Gamma)$ due to Proposition~\ref{prop-equiv-def}. Hence, according to Theorem~\ref{th-basic-series} and since $h$ decreases, we have the estimates
\begin{equation}\label{mean-degree-est-lambda-proof}
\biggl\|f-\sum_{j:\lambda_{j}'\leq1+\lambda^{2}}
(f,e_j)_{H}\,e_j\biggr\|_{W,\ell,p}\leq c_{0}\cdot\|g\|_{H}\cdot
\eta(1+\lambda^{2})
\end{equation}
and
\begin{equation}\label{mean-degree-est-proof}
\biggl\|f-\sum_{j=1}^{k}(f,e_j)_{H}\,e_j\biggr\|_{W,\ell,p}
\leq c_{0}\cdot\|g\|_{H}\cdot\eta(\lambda_{k}'),
\end{equation}
where $c_{0}$ is the norm of $\omega(L)$ considered as a bounded operator from $H=L_{2}(\Gamma)$ to $N=W^{\ell}_{p}(\Gamma)$ (see the proof of Theorem~\ref{th-mean}), $g:=((\omega\eta)(L))^{-1}f$, and the number $\lambda_{j}'$ satisfies $Le_{j}=\lambda_{j}'e_{j}$, i.e. $\lambda_{j}'=1+|\lambda_{j}|^{2}$. Hence, $\|g\|_{H}\asymp\|f\|_{H,\alpha}$, $\eta(1+\lambda^{2})\asymp h(\lambda^{1/m})$ as $\lambda\geq1$, and
\begin{equation*}
\eta(\lambda_{k}')=h((1+|\lambda_{k}|^{2})^{1/(2m)})\asymp h(k^{1/n})
\quad\mbox{as}\quad k\geq1,
\end{equation*}
due to $h\in\mathrm{OR}$ and \eqref{eigenvalue-asymptotic}. (As usual, the symbol $\asymp$ means the weak equivalence of positive values.) Thus, \eqref{mean-degree-est-lambda-proof} and \eqref{mean-degree-est-proof} imply the required estimates \eqref{mean-degree-est-lambda} and \eqref{mean-degree-est}, resp.
\end{proof}

\begin{remark}\label{norm-H-alpha}
If $A$ is a positive definite operator in $L_{2}(\Gamma)$, we have the equivalence of norms
\begin{equation*}
\|f\|_{H,\alpha}\asymp\biggl(\,\sum_{j=1}^{\infty}
\alpha^{2}(j^{1/n})\,|(f,e_j)_{H}|^{2}\biggr)^{1/2}\asymp
\biggl(\,\sum_{j=1}^{\infty}
\alpha^{2}((1+\lambda_{j})^{1/m})\,|(f,e_j)_{H}|^{2}\biggr)^{1/2},
\end{equation*}
which follows from Proposition~\ref{prop-equiv-def} (cf. \cite[Theorem~2.7]{MikhailetsMurach14}).
\end{remark}

Let us consider three examples of a function $\alpha$ that satisfy hypotheses of Theorem~\ref{th-mean-degree}. As above,  $0\leq\ell\in\mathbb{Z}$ and $2<p\in\mathbb{R}$. We arbitrarily choose a number $\varepsilon>0$.

Putting $\alpha(t):=t^{l+\varepsilon+n/2-n/p}$ whenever $t\geq1$, we obtain the power estimates \eqref{mean-degree-est-lambda} and \eqref{mean-degree-est} with $h(\lambda^{1/m})=\lambda^{-\varepsilon/m}$ and $h(k^{1/n})=k^{-\varepsilon/n}$, resp. In this case, $H^{\alpha}(\Gamma)$ is a Sobolev space.

Taking $s:=l+n/2-n/p$ and $r:=s+\varepsilon$ in \eqref{example-gamma} and defining $\alpha$ by formula \eqref{OR-description} with $\beta(t)\equiv1$, we receive the H\"ormander space $H^{\alpha}(\Gamma)$ that is not a part of the union of all Sobolev spaces $W^{s}_{2}(\Gamma)$ such that $s>\ell+n/2-n/p$. In this case, $h(t)\searrow0$ as $t\to\infty$.

The third example is given by any function $\alpha\in\mathrm{OR}$ such that
\begin{equation*}
\alpha(t):=t^{l+n/2-n/p}\,
(\underbrace{\log\ldots\log}_{k\;\mbox{\tiny{times}}} t)^{\varepsilon}
\quad\mbox{whenever}\;\;t\gg1.
\end{equation*}
In this case, the estimate \eqref{mean-degree-est} is of a logarithmic kind, and the space $H^{\alpha}(\Gamma)$ is broader than the above union of  Sobolev spaces.

\begin{remark}
Theorems \ref{th-mean} and \ref{th-mean-degree} remain valid for every fractional $\ell>0$, which follows from their proofs.
\end{remark}

\begin{remark}
Let $0\leq\ell\in\mathbb{R}$, $1\leq q<2<p$, and $s=\ell+n/q-n/p$. Assertions (ii) and (iii) of Theorem~\ref{th-mean} hold true if we replace the normed space $W^{\ell}_{p}(\Gamma)$ with the Triebel--Lizorkin space $F^{\ell}_{p,\theta}(\Gamma)$, where $1\leq\theta\leq\infty$, or with the Besov space $B^{\ell}_{p,\theta}(\Gamma)$, where $2\leq\theta\leq\infty$, and replace the convergence class $W^{s}_{q}(\Gamma)$
with $F^{s}_{q,\eta}(\Gamma)$, where $1\leq\eta\leq\infty$, or with
$B^{s}_{q,\eta}(\Gamma)$, where $1\leq\eta\leq2$. This follows from the continuous embeddings
\begin{equation*}
F^{s}_{q,\eta}(\Gamma)\hookrightarrow
W^{\ell+n/2-n/p}_{2}(\Gamma)\hookrightarrow
F^{\ell}_{p,\theta}(\Gamma)
\end{equation*}
and
\begin{equation*}
B^{s}_{q,\eta}(\Gamma)\hookrightarrow
W^{\ell+n/2-n/p}_{2}(\Gamma)\hookrightarrow
B^{\ell}_{p,\theta}(\Gamma)
\end{equation*}
\cite[p.~60, properties (ii) and (iii)]{Triebel06}, as we can conclude analyzing the proof of this theorem. Of course, it is natural to restrict ourselves to the marginal cases where $\eta=\infty$ and $\theta=1$ for $F$-spaces and where $\eta=\theta=2$ for $B$-spaces. In these cases we obtain narrower normed spaces than $W^{\ell}_{p}(\Gamma)$ and get broader convergence classes than $W^{s}_{q}(\Gamma)$. Of course, Theorem~\ref{th-mean-degree} is also valid under this replacement of $W^{\ell}_{p}(\Gamma)$.
\end{remark}

\section{Uniform convergence}\label{sec5}

We use this term to refer to the convergence in the normed space $C^{\ell}(\Gamma)$ of $\ell$ times continuously differentiable functions on the manifold $\Gamma$, with $0\leq\ell\in\mathbb{Z}$. Specifically, if $\ell=0$, we get the convergence in the space $C(\Gamma)$ of continuous functions on $\Gamma$, i.e. the uniform convergence on $\Gamma$. Let $\|\cdot\|_{C,\ell}$ denote the norm in $C^{\ell}(\Gamma)$. If $\ell\geq1$, we use the norm
\begin{equation*}
\|f\|_{C,\ell}:=\sum_{j=1}^{\varkappa}\,\max_{|\varrho|\leq\ell}\,
\sup_{y\in\mathbb{R}^{n}}|\partial^{\varrho}((\chi_{j}f)\circ\pi_{j})(y)|,
\end{equation*}
with $\pi_{j}$ and $\chi_{j}$ being taking from the definition of H\"ormander spaces over $\Gamma$. Here, of course, $\varrho=(\varrho_1,\ldots,\varrho_n)$ is a multi-index, $|\varrho|=\varrho_1+\cdots+\varrho_n$, and $\partial^{\varrho}$ is the partial derivative corresponding to $\varrho$.

\begin{theorem}\label{th-uniform}
Suppose that $0<m\in\mathbb{R}$, $0\leq\ell\in\mathbb{Z}$, and $\alpha\in\mathrm{OR}$. Then the following three conditions are equivalent:
\begin{itemize}
  \item[(i)]
\begin{equation*}
\int\limits_{1}^{\infty}\frac{t^{2\ell+n-1}}{\alpha^{2}(t)}\,dt<\infty;
\end{equation*}
  \item[(ii)] there exists a PsDO $A\in\mathrm{NE}\Psi_{\mathrm{ph}}^{m}(\Gamma)$ that the spectral decomposition \eqref{spectral-decomp} induced by $A$ converges in $C^{\ell}(\Gamma)$ on the class $H^{\alpha}(\Gamma)$;
  \item[(iii)] the expansions in eigenfunctions of an arbitrary PsDO from $\mathrm{NE}\Psi_{\mathrm{ph}}^{m}(\Gamma)$ converge unconditionally in $C^{\ell}(\Gamma)$ on the class $H^{\alpha}(\Gamma)$.
\end{itemize}
\end{theorem}

\begin{proof}
It follows from  H\"ormander's embedding theorem \cite[Theorem 2.2.7]{Hermander63} that
\begin{equation}\label{Hormander-embedding-th}
\mbox{condition (i)}\;\Longleftrightarrow\;
H^{\alpha}(\Gamma)\subset C^{\ell}(\Gamma),
\end{equation}
as is seen from \cite[Proposition 2.6~(vi)]{MikhailetsMurach14}.

Let now prove that $\mathrm{(i)}\Rightarrow\mathrm{(iii)}$. We assume (i) to hold true and arbitrarily choose a PsDO $A\in\mathrm{NE}\Psi_{\mathrm{ph}}^{m}(\Gamma)$. Owing to \eqref{Hormander-embedding-th}, the space $H^{\alpha}(\Gamma)$ lies in $L_{2}(\Gamma)$. Hence, the function $1/\alpha$ is bounded on $[1,\infty)$ in view of \cite[Theorem 2.2.2]{Hermander63}. We put $\omega(t):=1/\alpha(t^{1/(2m)})$ whenever $t\geq1$ and consider the function $\omega(L)$ of the operator $L=(I+A^{\ast}A)\in\mathrm{NE}\Psi_{\mathrm{ph}}^{2m}(\Gamma)$ acting in $L_{2}(\Gamma)$. According to Proposition~\ref{prop-equiv-def}, the operator $\omega(L)$ sets an isomorphism between $L_{2}(\Gamma)$ and $H^{\alpha}(\Gamma)$. Hence, owing to \eqref{Hormander-embedding-th}, this operator acts continuously from $H=L_{2}(\Gamma)$ to $N:=C^{\ell}(\Gamma)$, as noticed in Remark~\ref{rem-comp-norms}. We therefore conclude by Theorem~\ref{th-basic-series}, where $\eta(t)\equiv1$, that the expansions in eigenfunctions of $L$ converge unconditionally in $C^{\ell}(\Gamma)$ on the class $H^{\alpha}(\Gamma)$.
We have proved that $\mathrm{(i)}\Rightarrow\mathrm{(iii)}$.

The implication $\mathrm{(iii)}\Rightarrow\mathrm{(ii)}$ is obvious.

The implication $\mathrm{(ii)}\Rightarrow\mathrm{(i)}$ is true because assertion (ii) entails the inclusion $H^{\alpha}(\Gamma)\subset C^{\ell}(\Gamma)$ and hence implies (i) by \eqref{Hormander-embedding-th}.    \end{proof}

Let $m>0$ and $A\in\mathrm{NE}\Psi_{\mathrm{ph}}^{m}(\Gamma)$. As in the previous section, $e:=(e_{j})_{j=1}^{\infty}$ is an orthonormal basis in $H=L_{2}(\Gamma)$ formed by eigenfunctions $e_{j}$ of $A$, with $Ae_j=\lambda_{j}e_j$ and $|\lambda_{j}|\leq|\lambda_{j+1}|$ whenever $j\geq1$. We  complement Theorem~\ref{th-uniform} by estimating the degree of the uniform convergence of the series \eqref{f-series-separable} on H\"ormander classes.

\begin{theorem}\label{th-uniform-degree}
Let $0\leq\ell\in\mathbb{Z}$, and suppose that certain functions $h,\beta\in\mathrm{OR}$ satisfy the following conditions:
$h(t)\to0$ as $t\to\infty$, $h$ decreases on $[1,\infty)$, and
\begin{equation*}
\int\limits_{1}^{\infty}\frac{t^{2\ell+n-1}}{\beta^{2}(t)}\,dt<\infty.
\end{equation*}
Put $\alpha:=\beta/h$, and note that the function $\alpha$ belongs to $\mathrm{OR}$ and satisfies hypothesis \rm(i) \it of Theorem~$\ref{th-uniform}$. Then
\begin{equation}\label{uniform-degree-est-lambda}
\biggl\|f-\sum_{j:|\lambda_{j}|\leq\lambda}
(f,e_j)_{H}\,e_j\biggr\|_{C,\ell}
\leq c\cdot\|f\|_{H,\alpha}\cdot h(\lambda^{1/m})
\end{equation}
and
\begin{equation}\label{uniform-degree-est}
\biggl\|f-\sum_{j=1}^{k}(f,e_j)_{H}\,e_j\biggr\|_{C,\ell}
\leq c\cdot\|f\|_{H,\alpha}\cdot h(k^{1/n})
\end{equation}
for all $f\in H^{\alpha}(\Gamma)$, $\lambda\geq1$, and integer-valued $k\geq1$. Here, $c$ is a certain positive number that does not depend on $f$, $\lambda$, and $k$.
\end{theorem}

\begin{proof}
Let, as above, $L:=(I+A^{\ast}A)$, and consider the bounded functions  $\omega(t):=(\beta(t^{1/(2m)}))^{-1}$ and $\eta(t):=h(t^{1/(2m)})$ of $t\geq1$. (As to $\omega$ note that every function of class $\mathrm{OR}\cap L_{1}[1,\infty)$ is bounded on $[1,\infty)$). Since $(\omega\eta)(t)\equiv(\alpha(t^{1/(2m)}))^{-1}$, the operator $(\omega\eta)(L)$ sets an isomorphism between $H=L_{2}(\Gamma)$ and $H^{\alpha}(\Gamma)$ due to \cite[Theorem~2.23]{MikhailetsMurach14}. Hence, according to Theorem~\ref{th-basic-series} and since the function $h$ decreases, we obtain the estimates of the form \eqref{mean-degree-est-lambda-proof} and \eqref{mean-degree-est-proof}, with the norm $\|\cdot\|_{C,\ell}$ being instead of $\|\cdot\|_{W,\ell,p}$. These estimates imply \eqref{uniform-degree-est-lambda} and \eqref{uniform-degree-est} with the help of the same reasoning as that given in the proof of Theorem~\ref{th-mean-degree}.
\end{proof}

Let us indicate important examples of the function $\alpha$ satisfying hypotheses of Theorem~\ref{th-uniform-degree}. We let $0\leq\ell\in\mathbb{Z}$ and arbitrarily choose numbers $\varepsilon$ and $\delta$ such that $0<\delta<\varepsilon$.

Putting $\alpha(t):=t^{l+n/2+\varepsilon+\delta}$ whenever $t\geq1$, we obtain the power estimates \eqref{uniform-degree-est-lambda} and  \eqref{uniform-degree-est} with $h(t)\equiv t^{-\delta}$. Such power estimates hold true for functions
\begin{equation*}
f\in\bigcup_{s>l+n/2}W^{s}_{2}(\Gamma)\subsetneqq W^{l+n/2}_{2}(\Gamma).
\end{equation*}

To achieve the limiting value $s=l+n/2$, we put $\alpha(t):=t^{l+n/2}\log^{1/2+\varepsilon+\delta}(t+1)$ whenever $t\geq1$ and receive the estimates of a logarithmic kind with $h(t)\equiv\log^{-\delta}(t+1)$.

\section{Applications to ordinary differential operators}\label{sec6}

Let us discuss results of Sections \ref{sec4} and \ref{sec5} in the case where $\Gamma$ is a circle $\mathbb{T}$ of the unit radius and when $A$ is an (ordinary) differential operator on $\mathbb{T}$ of order $m$ with infinitely smooth complex-valued coefficients. Let $\tau$, with $0\leq\tau\leq 2\pi$, set a natural parametrization of $\mathbb{T}$. We suppose that the leading coefficient of $A$ does not equal zero for any $\tau$, which is equivalent to the ellipticity of $A$ on $\mathbb{T}$. We also assume that $A$ is a normal operator in the Hilbert space $H:=L_{2}(\mathbb{T})$. If all coefficients of $A$ are constant, this assumption holds true. In the case of variables coefficients, we recall the following necessary and sufficient condition for $A$ to be self-adjoint in $L_{2}(\mathbb{T})$ \cite[Chapter~I, Section~1, Subsection~5]{Naimark67}: $A$ is a sum of differential operators of the form
\begin{equation*}
A_{2k}u:=(\mu_{k}u^{(k)})^{(k)}\quad\mbox{and}\quad
A_{2k-1}u:=\frac{i}{2}\bigl((\nu_{k}u^{(k)})^{(k-1)}+
(\nu_{k}u^{(k-1)})^{(k)}\bigr)
\end{equation*}
provided that $\mu_{k}$ and $\nu_{k}$ are real-valued functions of class $C^{\infty}(\mathbb{T})$. Specifically, if all coefficients of $A$ are real-valued, the above condition means that $m$ is even and that $A$ is the sum of $A_{2k}$ with $k=0,1,\ldots,m/2$.

Let $e:=(e_{j})_{j=1}^{\infty}$ be an orthonormal basis in $H=L_{2}(\mathbb{T})$ formed by eigenfunctions $e_{j}$ of $A$, with $Ae_j=\lambda_{j}e_j$ and $|\lambda_{j}|\leq|\lambda_{j+1}|$ whenever $j\geq1$. Given a function $f\in L_{2}(\mathbb{T})$, we consider its expansion
\begin{equation}\label{expansion-circle}
f(\tau)=\sum_{j=1}^{\infty}(f,e_j)_{H}\,e_j(\tau),
\quad\mbox{with}\quad 0\leq\tau\leq2\pi,
\end{equation}
in this eigenfunctions, the expansion converging in $L_{2}(\mathbb{T})$. Let $0\leq\ell\in\mathbb{Z}$ and $2<p\in\mathbb{R}$. If $1<q\leq2$ and $s\geq\ell+1/q-1/p$, then Theorem~\ref{th-mean} implies that, for every $f\in W^{s}_{q}(\mathbb{T})$, this expansion converges in the $p$-th mean  unconditionally and preserves this convergence after the termwise differentiation up to $\ell$ times. If a function parameter $\alpha\in\mathrm{OR}$ satisfies
\begin{equation*}
\int\limits_{1}^{\infty}\frac{t^{2\ell}}{\alpha^{2}(t)}\,dt<\infty,
\end{equation*}
then Theorem~\ref{th-uniform} implies that, for every $f\in H^{\alpha}(\mathbb{T})$, the expansion converges uniformly and unconditionally and preserve this convergence after the same differentiation. Note that the derivatives of $e_j(\tau)$ may not form an orthogonal basis in $L_{2}(\mathbb{T})$. Assuming $\alpha$ to satisfy hypotheses of Theorem~\ref{th-mean-degree} or Theorem~\ref{th-uniform-degree}, we see that the estimates \eqref{mean-degree-est-lambda} and \eqref{uniform-degree-est-lambda} become
\begin{equation*}
\sum_{k=0}^{\ell}\biggl(\,\int\limits_{0}^{2\pi}\,
\biggl|f^{(k)}(\tau)-\sum_{j:|\lambda_{j}|\leq\lambda}
(f,e_j)_{H}\,e_j^{(k)}(\tau)\biggr|^{p}d\tau\biggr)^{1/p}
\leq c\cdot\|f\|_{H,\alpha}\cdot h(\lambda^{1/m})
\end{equation*}
and
\begin{equation*}
\sum_{k=0}^{\ell}\,\sup_{0\leq\tau\leq2\pi}\,
\biggl|f^{(k)}(\tau)-\sum_{j:|\lambda_{j}|\leq\lambda}
(f,e_j)_{H}\,e_j^{(k)}(\tau)\biggr|
\leq c\cdot\|f\|_{H,\alpha}\cdot h(\lambda^{1/m})
\end{equation*}
whenever $f\in H^{\alpha}(\mathbb{T})$ and $\lambda\geq1$; here, $h(\lambda^{1/m})\to0$ as $\lambda\to\infty$. Formulas \eqref{mean-degree-est} and \eqref{uniform-degree-est} are rewritten quite similarly.

Considering Remark~\ref{norm-H-alpha} in the $A:=1-d^{2}/d^{2}\tau$ case, we conclude that the norm $\|f\|_{H,\alpha}$ in the space $H^{\alpha}(\mathbb{T})$ is equivalent to the norm
\begin{equation*}
\biggl(|c_{0}(f)|^{2}+\sum_{j=1}^{\infty}
\alpha^{2}(|j|)\,\bigl(|c_{j}(f)|^{2}+|c_{-j}(f)|^{2}\bigr)\biggr)^{1/2},
\end{equation*}
with $c_{j}(f)$ being the Fourier coefficient of $f$ with respect to the eigenfunction $e_{j}(\tau):=(2\pi)^{-1}e^{ij\tau}$ of the differential operator $\nobreak{1-d^{2}/d^{2}\tau}$.

In this case or if $A=d/d\tau$, the expansion \eqref{expansion-circle} becomes the Fourier series with respect to the complex trigonometric system on $[0,2\pi]$. The above results on the mean convergence are new even for this series.

Ending this section, we remark that Theorems \ref{th-mean}, \ref{th-mean-degree}, \ref{th-uniform}, and \ref{th-uniform-degree} remain valid in the case where the coefficients of the ordinary differential operator $A$ have finite smoothness depending, of course, on the parameters involved in these theorems. We will not specify the smoothness.

\section{Applications to multiple Fourier series}\label{sec7}

Let $2\leq n\in\mathbb{Z}$, and consider the $n$-dimensional torus $\mathbb{T}^{n}$; as above, $\mathbb{T}$ is a circle of the unit radius. Let $A$ denote the Laplace--Beltrami operator on $\mathbb{T}^{n}$. It belongs to $\mathrm{NE}\Psi_{\mathrm{ph}}^{2}(\Gamma)$, is a self-adjoint operator in the Hilbert space $H:=L_{2}(\mathbb{T}^{n})$, and takes the form $A=\partial^{2}/\partial^{2}\tau_{1}+\cdots+
\partial^{2}/\partial^{2}\tau_{n}$, where $\tau_{k}$, with $0\leq\tau_{k}\leq 2\pi$, sets a natural parametrization of the $k$-th specimen of $\mathbb{T}$. The eigenfunctions  $e_{j}(\tau):=(2\pi)^{-n}e^{i(j_1\tau_1+\cdots+j_n\tau_n)}$ of $A$, where $\tau=(\tau_1,\ldots,\tau_n)\in[0,2\pi]^{n}$ and $j=(j_1,\ldots,j_n)\in\mathbb{Z}^{n}$, form an orthonormal basis in the Hilbert space $L_{2}(\mathbb{T}^{n})$. Given a function $f\in L_{2}(\mathbb{T}^{n})$, we consider its expansion in the multiple Fourier series
\begin{equation}\label{expansion-torus}
f(\tau)=\sum_{j\in\mathbb{Z}^{n}}c_j(f)\,e_j(\tau),
\quad\mbox{with}\quad\tau\in[0,2\pi]^{n},
\end{equation}
the expansion converging in $L_{2}(\mathbb{T}^{n})$. Here, $c_j(f)$ denotes the Fourier coefficient of $f$ with respect to $e_{j}$,

Since the termwise differentiation of the series \eqref{expansion-torus} gives the multiple Fourier series again, we restrict ourselves to the $\ell=0$ case applying results of Sections \ref{sec4} and \ref{sec5} to the expansion \eqref{expansion-torus}. Let $2<p<\infty$. If $1<q\leq2$ and $s\geq n/q-n/p$, it follows from Theorem~\ref{th-mean} that, whatever $f\in W^{s}_{q}(\mathbb{T}^{n})$, the expansion \eqref{expansion-torus} is unconditionally convergent in the $p$-th mean. If a function parameter $\alpha\in\mathrm{OR}$ satisfies
\begin{equation}\label{cond-alpha-unif-conv-torus}
\int\limits_{1}^{\infty}\frac{t^{n-1}}{\alpha^{2}(t)}\,dt<\infty,
\end{equation}
it follows from Theorem~\ref{th-uniform} that, for every $f\in H^{\alpha}(\mathbb{T}^{n})$, this expansion converges unconditionally uniformly (on $\mathbb{T}^{n}$). Considering the last result in the case of power functions $\alpha(t)\equiv t^{s}$, we conclude that the expansion \eqref{expansion-torus} converges unconditionally uniformly on each class $W^{s}_{2}(\mathbb{T}^{n})$ where $s>n/2$. If $n$ is odd, such a conclusion is substantiated in \cite[Theorem~A.1]{Ilin95} concerning the uniform convergence over balls.

Assuming $\alpha$ to satisfy hypotheses of Theorem~\ref{th-mean-degree} or Theorem~\ref{th-uniform-degree}, we rewrite the estimates \eqref{mean-degree-est-lambda} and \eqref{uniform-degree-est-lambda} as follows:
\begin{equation*}
\biggl(\,\,\int\limits_{[0,2\pi]^{n}}\,
\biggl|f(\tau)-\sum_{j\in\mathbb{Z}^{n}:\|j\|\leq\lambda}
c_j(f)\,e_{j}(\tau)\biggr|^{p}d\tau\biggr)^{1/p}
\leq c\cdot\|f\|_{H,\alpha}\cdot h(\lambda)
\end{equation*}
and
\begin{equation*}
\sup_{\tau\in[0,2\pi]^{n}}\,
\biggl|f(\tau)-\sum_{j\in\mathbb{Z}^{n}:\|j\|\leq\lambda}
c_j(f)\,e_{j}(\tau)\biggr|
\leq c\cdot\|f\|_{H,\alpha}\cdot h(\lambda),
\end{equation*}
whenever $f\in H^{\alpha}(\mathbb{T})$ and $\lambda\geq1$. Here, $h(\lambda)\to0$ as $\lambda\to\infty$, and we take into account that $Ae_{j}=-\|j\|^{2}e_{j}$, with $\|j\|:=(j_{1}^{2}+\cdots+j_{n}^{2})^{1/2}$. According to Remark~\ref{norm-H-alpha}, the norm $\|f\|_{H,\alpha}$ in $H^{\alpha}(\mathbb{T}^{n})$ is equivalent to the norm
\begin{equation*}
\biggl(|c_{0}(f)|^{2}+\sum_{j\in\mathbb{Z}^{n},j\neq0}
\alpha^{2}(\|j\|)\,|c_{j}(f)|^{2}\biggr)^{1/2}.
\end{equation*}

If $\alpha$ belongs to $\mathrm{OR}$ and satisfies \eqref{cond-alpha-unif-conv-torus}, then
\begin{equation}\label{absolute-convegence-multiple-Fourier-series}
\sum_{j\in\mathbb{Z}^{n}}|c_{j}(f)|<\infty
\end{equation}
for every $f\in H^{\alpha}(\mathbb{T}^{n})$, which means the absolute convergence of the Fourier series \eqref{expansion-torus} of $f$ and follows from the unconditional uniform convergence of this series. As is known \cite[Section~6, Subsection~3]{AlimovIlinNikishin76}, property \eqref{absolute-convegence-multiple-Fourier-series} is fulfilled provided that $f$ belongs to the H\"older space $\mathcal{C}^{\ell}(\mathbb{T}^{n})$ for some $\ell>n/2$ and is not fulfilled for certain functions $f\in\mathcal{C}^{n/2}(\mathbb{T}^{n})$. Our sufficient condition for \eqref{absolute-convegence-multiple-Fourier-series} to hold true is weaker than that indicated in terms of H\"older spaces. This is demonstrated by the example when $\alpha(t)\equiv t^{n/2}\log^{1/2+\varepsilon}(t+1)$ for some $\varepsilon>0$, which gives the broader space $H^{\alpha}(\mathbb{T}^{n})$ (satisfying \eqref{cond-alpha-unif-conv-torus}) than the union of all $\mathcal{C}^{\ell}(\mathbb{T}^{n})$ with $\ell>n/2$.

\section{Concluding Remarks}\label{sec8}

Our results are also applied to the case where the operator $A\in\mathrm{NE}\Psi_{\mathrm{ph}}^{m}(\Gamma)$ is of negative order $m$ because the resolvent of $A$ belongs to $\mathrm{NE}\Psi_{\mathrm{ph}}^{-m}(\Gamma)$, is of positive order, and has the same eigenfunctions as $A$ does.

If the closed manifold $\Gamma$ is a sphere (of any dimension), the results of Sections \ref{sec4} and \ref{sec5} admit further detailing and are also new.

Our approach allows other applications, specifically, to elliptic matrix pseudodifferential operators and elliptic differential operators in a bounded Euclidean domain.

\subsection*{Acknowledgments.} This work was funded by the Czech Academy of Sciences within grant RVO:67985840 and by the National Academy of Sciences of Ukraine. The authors was supported by the European Union's Horizon 2020 research and innovation programme under the Marie Sk{\l}odowska-Curie grant agreement No~873071 (SOMPATY: Spectral Optimization: From Mathematics to Physics and Advanced Technology). The second named author was also supported by a grant from the Simons Foundation (1030291, A.A.M.).

\end{document}